\documentclass[12pt]{amsart}
\def\N{{\mathbb{N}}}
\def\Z{{\mathbb{Z}}}
\def\R{{\mathbb{R}}}
\def\S{{\mathbb{S}}}

\def\Int{\operatorname{Int}}

\newtheorem{Theorem}{Theorem}[section]

\theoremstyle{definition}

\newtheorem*{Question}{Question}
\theoremstyle{remark}

\newtheorem*{Acknowledgement}{Acknowledgement}
\begin{document}
\sloppy
\title{On boundaries of Coxeter groups and topological fractal structures}
\author{Tetsuya Hosaka} 
\address{Department of Mathematics, Faculty of Education, 
Utsunomiya University, Utsunomiya, 321-8505, Japan}
\date{November 29, 2010}
\email{hosaka@cc.utsunomiya-u.ac.jp}
\thanks{
Partly supported by the Grant-in-Aid for Young Scientists (B), 
The Ministry of Education, Culture, Sports, Science and Technology, Japan.
(No.\ 21740037).}
\maketitle
\begin{abstract}
In this paper, 
based on research on rank-one isometries by W.~Ballmann and M.~Brin and 
recent research on rank-one isometries of Coxeter groups 
by P.~Caprace and K.~Fujiwara, 
we study a topological fractal structure of boundaries of Coxeter groups.
We also show that 
the limit-point set is dense in a boundary of a Coxeter group 
and introduce some observations 
on boundaries of CAT(0) groups with rank-one isometries.
\end{abstract}

\section{Introduction}

In this paper, we study boundaries of Coxeter groups, 
where we suppose that Coxeter groups are finitely generated and infinite.
A Coxeter group acts geometrically (i.e.\ properly and cocompactly by isometries) 
on a Davis complex which is a CAT(0) space \cite{M} and 
every Coxeter group is a CAT(0) group.
Details of Coxeter groups and Coxeter systems 
are found in \cite{Bo}, \cite{Br}, \cite{De}, \cite{Hu} and \cite{T}, and 
details of CAT(0) spaces, CAT(0) groups and their boundaries 
are found in \cite{BH}, \cite{CK} and \cite{GH}.

Now we suppose that an infinite group $G$ acts geometrically on a proper CAT(0) space $X$ 
and $G$ is non-elementary (hence $|\partial X|>2$).

A hyperbolic isometry $g$ of a proper CAT(0) space $X$ is 
said to be {\it rank-one}, 
if some (any) axis for $g$ does not bound a flat half-plane.
In \cite[Theorem~A]{BB}, W.~Ballmann and M.~Brin have proved that 
if there exists a rank-one isometry $g\in G$ of $X$ 
then 
for any two non-empty open subsets $U$ and $V$ of $\partial X$, 
there exists an element $g \in G$ 
such that $g (\partial X - U)\subset V$ 
and $g^{-1} (\partial X - V)\subset U$ 
where it is possible to choose $g$ to be rank-one 
(cf.\ \cite{CF}, \cite{Ha}).

This statement implies that 
if there exists a rank-one isometry $g\in G$ of $X$ 
then we can say that 
the boundary $\partial X$ has a {\it topological fractal structure}; 
that is, 
for any proper closed subset $F$ of $\partial X$ and 
any non-empty open subset $U$ of $\partial X$, 
there exists $g \in G$ such that $g F\subset U$.

We first note that 
if $G$ is hyperbolic then 
$G$ contains a rank-one isometry and 
the boundary $\partial X$ has a topological fractal structure.

In particular, 
if $G$ is hyperbolic and the boundary $\partial X$ is an $n$-sphere 
then the boundary $\partial X \approx \S^n$ 
has a topological fractal structure.
This case is the most simple case of boundaries of CAT(0) groups with rank-one isometries.
In general, the boundary $\partial X$ with a topological fractal structure 
is very complex.

In \cite{F}, H.~Fischer has investigated the boundary $\partial \Sigma$ 
of the Davis complex of a right-angled Coxeter group 
whose nerve is a connected closed orientable PL-manifold.
These boundaries are typical examples of 
boundaries with topological fractal structures.
If the boundary $\partial X$ with a topological fractal structure 
contains some proper closed subset $F$ 
which has a something non-trivial homotopy type, 
then any (small) open subset $U$ of $\partial X$ 
contains $g F$ for some homeomorphism $g \in G$ of $\partial X$ 
and $\{gF\,|\,g\in G\}$ is dense in $\partial X$, 
where every $g F$ is homeomorphic to $F$.

Also for a proper closed subset $F$ of 
the boundary $\partial X$ with a topological fractal structure 
such that the complement $\partial X - F$ is a very small neighborhood, 
any (small) open subset $U$ of $\partial X$ 
contains $gF$ for some homeomorphism $g \in G$ of $\partial X$.

Thus, in such a case that $G$ contains a rank-one isometry and 
$\partial X$ is not an $n$-sphere, 
then the boundary $\partial X$ 
seems to be just a topological fractal.

This fractal structure seems to be suggested in some research 
on boundaries of CAT(0) groups by M.~Bestvina (cf.\ \cite{Bes2}) and 
some research on cohomology of boundaries of Coxeter groups 
(cf.\ \cite{Bes}, \cite{Dav}, \cite{Dr}, \cite{Ho0}).

If the boundary $\partial X$ has a topological fractal structure, 
then (the action of $G$ on) $\partial X$ is {\it minimal}; 
that is, every orbit $G \alpha$ is dense in the boundary $\partial X$.
Indeed if we take $F=\{\alpha\}$ then for any open subset $U$ of $\partial X$, 
$g F \subset U$ for some $g \in G$.

Also then (the action of $G$ on) $\partial X$ is {\it scrambled}; 
that is, for any two points $\alpha,\beta\in \partial X$ with $\alpha\neq \beta$, 
\begin{align*}
&\limsup\{d_{\partial X}(g\alpha,g\beta)\,|\,g\in G\}>0 \ \text{and} \ \\
&\liminf\{d_{\partial X}(g\alpha,g\beta)\,|\,g\in G\}=0 
\end{align*}
(cf.\ \cite{Ho2}).
Indeed 
$\limsup\{d_{\partial X}(g\alpha,g\beta)\,|\,g\in G\}>0$ 
always holds (\cite[Theorem~3.1]{Ho2}) and 
if we take $F=\{\alpha,\beta\}$ then 
for any small open subset $U$ of $\partial X$, 
$gF \subset U$ for some $g \in G$, 
hence 
$\liminf\{d_{\partial X}(g\alpha,g\beta)\,|\,g\in G\}=0$.

Thus if the boundary $\partial X$ is a topological fractal, 
then $\partial X$ is minimal and scrambled.

We can find recent research 
on minimality and scrambled sets of boundaries of Coxeter groups 
in \cite{Ho1} and \cite{Ho2}.

From recent research on rank-one isometries of Coxeter groups 
by P.~Caprace and K.~Fujiwara \cite[Proposition~4.5]{CF}, 
we obtain that 
for a Coxeter system $(W,S)$ such that 
$S$ is finite and $W$ is infinite and non-elementary, 
if $(W,S)$ is irreducible and non-affine 
then the Coxeter group $W$ 
contains a rank-one isometry of the Davis complex $\Sigma$ defined by $(W,S)$.
Hence 
a finitely generated, infinite and non-elementary Coxeter group $W$ 
contains a rank-one isometry 
if and only if $W$ does not contain 
a finite-index subgroup 
which splits as a product $W_1\times W_2$ 
where $W_1$ and $W_2$ are infinite.

By the observation above, 
we obtain the following theorem.

\begin{Theorem}\label{Thm1}
Let $(W,S)$ be a Coxeter system 
such that $W$ is infinite and non-elementary and $S$ is finite.
For the Davis complex $\Sigma$ of $(W,S)$ and 
any proper CAT(0) space $X$ on which $W$ acts geometrically, 
the following statements are equivalent.
\begin{enumerate}
\item[(1)] $(W_{\tilde{S}},\tilde{S})$ is irreducible and non-affine.
\item[(2)] $W$ contains a rank-one isometry of $\Sigma$.
\item[(3)] $W$ contains a rank-one isometry of $X$.
\item[(4)] $\partial \Sigma$ has a topological fractal structure.
\item[(5)] $\partial \Sigma$ is minimal.
\item[(6)] $\partial \Sigma$ is scrambled.
\item[(7)] $\partial X$ has a topological fractal structure.
\item[(8)] $\partial X$ is minimal.
\item[(9)] $\partial X$ is scrambled.
\item[(10)] $\Sigma$ does not contain a quasi-dense subspace 
which splits as a product $\Sigma_1\times \Sigma_2$ of two unbounded subspaces.
\item[(11)] $X$ does not contain a quasi-dense subspace 
which splits as a product $X_1\times X_2$ of two unbounded subspaces.
\item[(12)] $W$ does not contain a finite-index subgroup 
which splits as a product $W_1\times W_2$ of two infinite subgroups.
\end{enumerate}
\end{Theorem}

Here $W_{\tilde{S}}$ is the minimum finite-index parabolic subgroup of $(W,S)$ 
(\cite{De}, cf.\ \cite{Ho1}, \cite{Ho2}).

Thus if $(W,S)$ is an irreducible Coxeter system, then 
$W$ is finite, $W$ is affine or $W$ contains a rank-one isometry.

Hence for any Coxeter system $(W,S)$ and 
the irreducible decomposition of $(W,S)$ as 
$$ W=W_{S_1}\times \dots \times W_{S_k}\times W_{S_{k+1}} \times \dots \times W_{S_n}, $$
each $W_{S_i}$ is finite, affine or contains a rank-one isometry.

It is known that the following problem is open.

\begin{Question}
Suppose that a group $G$ acts geometrically on a proper CAT(0) space $X$.
Then is it the case that 
the limit-point set $\{g^\infty\,|\, g\in G ,\ o(g)=\infty\}$ 
is dense in the boundary $\partial X$?
\end{Question}

Here $g^\infty$ is the limit-point of the boundary $\partial X$ 
to which the sequence $\{g^ix_0\,|\,i\in \N\}\subset X$ converges in $X\cup \partial X$, 
where $x_0$ is a point of $X$ and the limit-point $g^\infty$ is not depend on $x_0$.
We note that any element $g$ of a CAT(0) group $G$ with the order $o(g)=\infty$ 
is a hyperbolic isometry.

We obtain a positive answer to this question for Coxeter groups.

\begin{Theorem}\label{Thm2}
Suppose that a finitely generated infinite Coxeter group $W$ 
acts geometrically on a proper CAT(0) space $X$.
Then the limit-point set $\{w^\infty\,|\, w\in W ,\ o(w)=\infty\}$ 
is dense in the boundary $\partial X$.
\end{Theorem}

Finally, we introduce some observations on 
boundaries of CAT(0) groups with rank-one isometries in Section~4, 
which relates to local properties of boundaries of CAT(0) groups.

\section{Rank-one isometries of Coxeter groups and 
topological fractal structures of their boundaries}

We prove Theorem~\ref{Thm1}.

\begin{proof}[Proof of Theorem~\ref{Thm1}]
We first obtain the equivalence $(1)\Leftrightarrow (2)\Leftrightarrow (12)$ 
from \cite[Proposition~4.5]{CF} and the observation in Section~1.
Also $(2)\Leftrightarrow (3)$ holds by \cite[Theorem~B]{BB}.

From the observation in Section~1 
on rank-one isometries and topological fractal structures of boundaries, 
we obtain $(2)\Rightarrow (4)$, 
$(4)\Rightarrow (5)$ and $(4)\Rightarrow (6)$, also, 
$(3)\Rightarrow (7)$,
$(7)\Rightarrow (8)$ and
$(7)\Rightarrow (9)$.

Concerning scrambled sets of boundaries, 
\cite[Theorem~5.5]{Ho2} implies 
$(6)\Rightarrow (10)$ 
and $(9)\Rightarrow (11)$.

Also concerning minimality of boundaries, 
\cite[Theorem~6.4]{Ho1} implies 
$(5)\Rightarrow (12)$ 
and $(8)\Rightarrow (12)$.

By splitting theorems (cf.\ \cite{Ho3}, \cite{Mo}), 
we obtain $(10)\Rightarrow (12)$ and $(11)\Rightarrow (12)$
(cf.\ \cite[Proposition~6.3]{Ho1}).

Therefore 
the statements $(1)$--$(12)$ are equivalent.
\end{proof}

\section{On limit-point sets of boundaries of Coxeter groups}

We prove Theorem~\ref{Thm2}.

\begin{proof}[Proof of Theorem~\ref{Thm2}]
Suppose that a finitely generated infinite Coxeter group $W$ 
acts geometrically on a proper CAT(0) space $X$.

Here there exists $S\subset W$ such that $(W,S)$ is a Coxeter system.
Now we consider the irreducible decomposition of $(W,S)$ as 
$$ W=W_{S_1}\times \dots \times W_{S_k}\times W_{S_{k+1}} \times \dots \times W_{S_n} $$
where each $(W_i,S_i)$ is irreducible and 
we may suppose that $W_{S_i}$ is infinite for any $i=1,\dots,k$ 
and $W_{S_i}$ is finite for any $i=k+1,\dots,n$.
Let $W'=W_{S_1}\times \dots \times W_{S_k}$.
Then $W'$ is a finite-index subgroup of $W$ and 
acts geometrically on the CAT(0) space $X$ 
(where $W'$ is the minimum finite-index parabolic subgroup of $(W,S)$).

Here we note that every Coxeter group has finite center.
Hence by the splitting theorem 
\cite[Theorem~2]{Ho3} and \cite[Corollary~10]{Mo}, 
$X$ contains a closed convex $W'$-invariant quasi-dense subspace $X'$ 
which splits as a product $X'=X_1 \times \dots \times X_k$ 
where the action of $W'=W_{S_1}\times \dots \times W_{S_k}$ 
on $X'=X_1 \times \dots \times X_k$ splits and 
$W_{S_i}$ acts geometrically on $X_i$ for each $i=1,\dots,k$.

Then every irreducible infinite Coxeter group $W_{S_i}$ 
is either affine or contains a rank-one isometry 
by \cite[Proposition~6.5]{CF} and the observation in Section~1.

If $W_{S_i}$ is affine, then 
$W_{S_i}$ contains a finite-index subgroup which isomorphic to $\Z^{n_i}$ 
and $X_i$ contains a quasi-dense subspace which isometric to $\R^{n_i}$.
Hence the limit-point set 
$\{w_i^\infty\,|\,w_i\in W_i,\ o(w_i)=\infty\}$ is dense in 
the boundary $\partial X_i$.

Also if $W_{S_i}$ contains a rank-one isometry, then 
the action of $W_{S_i}$ on the boundary $\partial X_i$ is minimal.
Hence \cite[Proposition~6.2]{Ho1} implies that 
the limit-point set $\{w_i^\infty\,|\,w_i\in W_i,\ o(w_i)=\infty\}$ is dense in 
the boundary $\partial X_i$.

Therefore, 
by a similar argument to the proof of \cite[Proposition~6.5]{Ho1}, 
we obtain that 
the limit-point set $\{w^\infty\,|\,w\in W,\ o(w)=\infty\}$ 
is dense in the boundary $\partial X$.
\end{proof}

\section{Observations on boundaries of CAT(0) groups with rank-one isometries}

We introduce some observations on 
boundaries of CAT(0) groups with rank-one isometries.

Now we suppose that 
a group $G$ acts geometrically on a proper CAT(0) space $X$ and 
suppose that $G$ contains a rank-one isometry 
(hence the boundary $\partial X$ has a topological fractal structure).

Let $V$ be a non-empty open subset of $\partial X$ 
whose closure ${\rm cl}\, V$ is a proper subset of $\partial X$.
Then there exists a rank-one isometry $g\in G$ as $g^\infty \in V$, 
because the limit-point set of rank-one isometries in $G$ is dense in $\partial X$.
Indeed $\partial X$ is minimal and 
$$G g^\infty=\{ag^\infty\,|\,a\in G\}=\{(aga^{-1})^\infty\,|\,a\in G\}$$ 
is dense in the boundary $\partial X$.

Every rank-one isometry acts 
with {\it north-south dynamics} on the boundary $\partial X$ (cf.\ \cite[p.7]{Ha}).
Hence, since $g$ is a rank-one isometry of $X$ and $g^\infty \in V$, 
the set $\{g^i V\,|\,i\in\N\}$ is a neighborhood basis for $g^\infty$ in $\partial X$.
Here all $g^i V$ are homeomorphic to $V$.

Thus if there exists a non-empty open subset $V$ of $\partial X$ 
whose closure ${\rm cl}\, V$ is a proper subset of $\partial X$ 
such that $V$ has some {\it topological property $(P)$}, 
then $\partial X$ has the locally {\it topological property $(P)$} 
at the limit-point $g^\infty$.

Also for any rank-one isometry $h \in G$, 
we can consider the limit-point $h^\infty \in \partial X$.
Then $G h^\infty$ is dense in $\partial X$, 
since $\partial X$ is minimal.
Hence $ah^\infty\in V$ for some $a\in G$.
Then $h^\infty\in a^{-1}V$ and $a^{-1}V$ is homeomorphic to $V$.
Thus 
the boundary $\partial X$ has the locally {\it topological property $(P)$} 
at the limit-point $h^\infty$ of all rank-one isometries $h \in G$.

We also note that the limit-point set of all rank-one isometries 
is dense in the boundary $\partial X$.

As one example, 
if there exists a non-empty {\it connected} open subset $V$ of $\partial X$ 
whose closure ${\rm cl}\, V$ is a proper subset of $\partial X$, 
then $\partial X$ is locally {\it connected} at 
the limit-points $g^\infty$ of all rank-one isometries $g\in G$.

Moreover if $\partial X$ is non-locally connected at some point $\alpha \in \partial X$, 
then 
$\partial X$ is non-locally connected at $g\alpha$ for all $g\in G$.
Here $G\alpha$ is also dense in $\partial X$.

It seems that 
these arguments relate to research on local connectivity of boundaries of CAT(0) groups 
by M.~Mihalik, K.~Ruane and S.~Tschantz (\cite{MR}, \cite{MRT}) 
and research on cut-points and limit-points of boundaries of CAT(0) groups 
by P.~Papasoglu and E.~L.~Swenson (\cite{PS}, \cite{Sw}).

Also as one application, 
we obtain the following theoreom by a similar argument to the proof of \cite[Theorem~4.4]{KB}.

\begin{Theorem}
If a CAT(0) group $G$ with a rank-one isometry 
acts geometrically on a proper CAT(0) space $X$, 
then the following statements are equivalent:
\begin{enumerate}
\item[{$\rm (i)$}] the boundary $\partial X$ is an $n$-manifold,
\item[{$\rm (ii)$}] the boundary $\partial X$ of $X$ contains 
some closed neighborhood $U$ which is homeomorphic to an $n$-ball,
\item[{$\rm (iii)$}] the boundary $\partial X$ is homeomorphic to an $n$-sphere.
\end{enumerate}
\end{Theorem}

\begin{proof}
We first note that the implications $\rm (iii)\Rightarrow (i) \Rightarrow (ii)$ are obvious.

Hence now we show the implication $\rm (ii)\Rightarrow (iii)$.

Suppose that $\rm (iii)$ holds; that is, the boundary $\partial X$ of $X$ contains 
some closed neighborhood $U$ which is homeomorphic to an $n$-ball.
For a point $\alpha \in \partial X -U$, 
there exists $g\in G$ such that $g\alpha \in \Int U$, 
since the action of $G$ on $\partial X$ is minimal.
Then $V:=g^{-1}U$ is a neighborhood of $\alpha$ 
which is homeomorphic to an $n$-ball.
Let $U'$ and $V'$ be a proper subsets of $\Int U$ and $\Int V$ respectively 
such that $U'$ and $V'$ are homeomorphic to an $n$-ball and $U' \cap V'=\emptyset$.
Let $F=\partial X - \Int U'$.
Then there exists $g'\in G$ such that $g'F \subset V'$, 
because the boundary $\partial X$ has a topological fractal structure.
Then $g'U'\cup V'=\partial X$ and $g'U'$ and $V'$ are homeomorphic to an $n$-ball.
(Moreover, $g'U\cup V=\partial X$ and $g'U$ and $V$ are homeomorphic to an $n$-ball.)
Using some argument on bicollars of $n$-disks, 
we obtain that $\partial X$ is homeomorphic to an $n$-sphere.
\end{proof}

\begin{Acknowledgement}
The author would like to thank Dr. Naotsugu Chinen for helpful discussion and helpful advice.
\end{Acknowledgement}

%

%
\end{document}